\documentclass[11pt]{amsart}
\usepackage[]{hyperref}
\usepackage{mathtools}
\usepackage{url}
\usepackage{a4wide}
\usepackage{amsthm}
\usepackage{amssymb}
\usepackage{palatino}
\usepackage{enumitem}
\usepackage{tcolorbox}
\usepackage{tikz}
\usetikzlibrary{positioning, arrows.meta}
\newcommand{\here}[2]{\tikz[remember picture]{\node[inner sep=0](#2){#1}}}
\usepackage{amssymb}
\newtheorem{thm}{Theorem}

\newtheorem{definition}{Definition}
\newtheorem{prop}{Proposition}
\newtheorem{cor}{Corollary}
\newtheorem{lemma}{Lemma}
\newtheorem{conjecture}{Conjecture}
\newtheorem{remark}{Remark}
\makeatletter

\newcommand{\Stab}{{\rm Stab}}
\newcommand*{\rom}[1]{\expandafter\@slowromancap\romannumeral #1@}

\def\blfootnote{\gdef\@thefnmark{}\@footnotetext}
\makeatother
\def\house#1{\setbox1=\hbox{$\,#1\,$}
\dimen1=\ht1 \advance\dimen1 by 2pt \dimen2=\dp1 \advance\dimen2 by 2pt
\setbox1=\hbox{\vrule height\dimen1 depth\dimen2\box1\vrule}%
\setbox1=\vbox{\hrule\box1}%
\advance\dimen1 by .4pt \ht1=\dimen1
\advance\dimen2 by .4pt \dp1=\dimen2 \box1\relax}
\def\ord{{\mathrm{ord}}}
\def\hol{{\mathrm{Hol}}}

\newcommand{\N}{{\mathbb N}}

\newcommand{\Z}{{\mathbb Z}}

\newcommand{\thmref}{Theorem~\ref}
\newcommand{\propref}{Proposition~\ref}

\newcommand{\defref}{Definition~\ref}
\newcommand{\lemref}{Lemma~\ref}

\newcommand{\secref}{Section~\ref}
\newcommand{\conjref}{Conjecture~\ref}

\begin{document}
\title{Davenport constant and its variants for some non-abelian groups}
\author{C. G. Karthick Babu, Ranjan Bera, Mainak Ghosh and B. Sury}
\address{Stat-Math Unit, Indian Statistical Institute, 8th Mile Mysore Road,
Bangalore 560059}
\email{cgkarthick24@gmail.com}
\email{ranjan.math.rb@gmail.com}
\email{mainak.09.13@gmail.com}
\email{surybang@gmail.com}






\maketitle

\begin{quote}\emph{Abstract.}  
We define two variants $e(G)$, $f(G)$ of the Davenport constant $d(G)$ of a finite group $G$, that is not necessarily abelian. These naturally arising constants aid in computing $d(G)$ and are of potential independent interest. We compute the constants $d(G)$, $e(G)$, $f(G)$ for some nonabelian groups G, and demonstrate that, unlike abelian groups where these constants are identical, they can each be distinct. As a byproduct of our results, we also obtain some cases of a conjecture of J. Bass. We compute the $k$-th Davenport constant for several classes of groups as well. We also make a conjecture on $f(G)$ for metacyclic groups and provide evidence towards it.
\end{quote}
\tableofcontents

\section{Introduction}

The Davenport constant arose in the context of class
groups of algebraic number fields where it gives information on the
number of prime ideals appearing as factors of any non-zero ideal.
For a finite abelian group $G$, it is the smallest positive integer
$d$ such that any sequence of elements (allowing repetition) of $G$
having length $d$, necessarily contains a subsequence
whose elements add to $0$. More generally, for a finite (not
necessarily abelian) group $G$, the {\it Davenport constant} $d(G)$
of a finite group $G$ (written multiplicatively) is the smallest
positive integer $d$ such that given any $d$-element sequence $S$ of
elements of $G$ (not necessarily distinct), it is possible to find some subsequence of $S$ whose terms have a product equal to the identity
when multiplied in some order. We remark that some
authors write $d(G)$ to mean $d(G)-1$ in our notation; that is, the length of the longest possible product-one free sequence.\\

One has $d(\Z_n) = n$, and also the trivial bound $d(G)
\le |G|$ for any $G$. The literature on Davenport constants for
abelian groups is vast.  Consider any finite abelian group
    \[G = \Z_{n_1} \oplus \Z_{n_2} \oplus \cdots \oplus \Z_{n_k}\]
    where $n_1\mid n_2\mid \cdots \mid n_{k-1}\mid n_k$.
If $M(G) := 1 + \displaystyle\sum_{i=1}^{k} (n_i - 1)$, then $d(G)
\ge M(G)$ as the sequence $S$, which contains $n_i - 1$ many
copies of $e_i$ for each $i$, contains no zero-sum subsequence, where 
\[
e_i=(0,\ldots,0,\here{1}{fromhere},0,\ldots,0).
\]
\begin{tikzpicture}[remember picture, overlay]
\node[font=\scriptsize, below right=12pt of fromhere] (tohere) {$i^{th}$-entry};
\draw ([yshift=-4pt]fromhere.south) |- (tohere);
\end{tikzpicture}

Olson proved (\cite{Ol1}, \cite{Ol2}) for rank two finite abelian
groups $G = \Z_m \oplus \Z_n$, where $m \mid n$ , and for finite
abelian $p$-groups $G = \Z_{p^{a_1}} \oplus \Z_{p^{a_2}} \oplus
\cdots \oplus \Z_{p^{a_k}}$, that $d(G) = M(G)$. Based on these, he
conjectured that this always holds for finite abelian groups.
However, Olson's conjecture does not hold in general, and there are
infinitely many rank $4$ finite abelian groups for which it fails. In fact, given any $m\in \N$, one can find a finite
abelian group $G$ for which $d(G) \ge M(G) + m$.


In the case of finite, non-abelian groups, Jared Bass
(\cite{JB07}) computed $d(G)$ for several classes of non-abelian
groups, and raised many questions. For instance, he proved that $d(\mathbf{D}_{2n}) = n+1$ and $d(\mathbf{Q}_{4n})=2n+1$. He also
considered
$$G_{a,b,g} :=  \Z_{a}\ltimes_g \Z_{b} = \left\langle x, y \mid x^b = y^a = 1, x^g y = yx \right\rangle,$$ with ${\rm ord} _{b}(g)=a$; he conjectured that
$a+b-1$ is the Davenport constant for this group. In the special
case where $q$ is prime, he proved this conjecture, showing
$d(G_{n,q,g}) = n+q-1$. Han and
Zhang proved (\cite{HZ}) that $d(\Z_{m}\ltimes \Z_{mn}) = mn + m -
1$.  Qu and Li made some progress towards Bass's conjecture, proving
(\cite{QL23}) that $d(G_{p,n,g}) = p + n - 1$, if $p$ is smaller
than every prime dividing $n$. Our results also include a proof of
Bass's conjecture for the more general case $\mathbb{Z}_p \ltimes
\mathbb{Z}_n$ for any prime $p$ and any positive integer $n$. \\

In this paper, while studying the Davenport constant for
certain non-abelian groups, some new variants naturally arose which
do not arise in the case of abelian groups. We believe these variants are of independent interest. We compute these variants for some groups and use them in determining the Davenport constant for certain other groups. 

Before we describe the statements of our main theorems, we need to set up some notations and definitions formally.

\subsection{Notations}\label{notation}
We define a \textit{sequence} of elements of a group $G$ as a tuple $S=(a_1, \ldots, a_n)$, where each $a_r \in G$. A \textit{subsequence} of $S$ is defined by taking a nonempty subset of the set of indices of $S$, say $\{i_1,\ldots, i_m\}\subset \{1,\ldots, n\}$, and constructing a sequence $T_1 = (a_{i_1},\ldots, a_{i_m})$. If $T_1$ is a proper subsequence of $S$, we define the deletion of $T_1$ from $S$, written as $S\setminus T_1$ to be the subsequence corresponding to the set of indices $\{1,\ldots, n\}\setminus\{i_1,\ldots, i_m\}$, in increasing order. 

Given another subsequence $T_2 = (a_{j_1},\ldots, a_{j_{\ell}})$ of $S$, we say that $T_1$ and $T_2$ are \textit{disjoint as subsequences} if the corresponding subsets of indices $\{i_1,\ldots, i_m\}$ and $\{j_1,\ldots, j_{\ell}\}$ are disjoint. If $T_1$ and $T_2$ are disjoint as subsequences, the \textit{concatenation} of $T_1$ and $T_2$, denoted $T_1 T_2$, is defined to be the subsequence $(a_{i_1},\ldots, a_{i_m}, a_{j_1},\ldots, a_{j_{\ell}})$. Given any subset $A\subseteq G$, we denote by $S\cap A$ to be the subsequence of $S$ obtained by deleting all the terms not contained in $A$.

For a group $G$ and for a sequence $S=(a_{1}, \ldots, a_{n})$ in it, we denote by $\Pi(S)$, the set of all products $a_{\sigma(1)} \cdots a_{\sigma(n)}$ as $\sigma$ varies over ${\rm Sym}(n)$. We call $S$ to be a \textit{product-one sequence} if $1\in \Pi(S)$. A subsequence $T$ of $S$ is called a \textit{product-one subsequence} of $S$ if $1\in \Pi(T)$. A sequence $S$ is called \textit{product-one free} if it does not contain a product-one subsequence.

We usually denote by $|S|$ the \textit{length} of a sequence $S$, and from the context, there will be no confusion with cardinality. Recall that the \textit{Davenport constant} for $G$, denoted $d(G)$, is defined to be the smallest positive integer $m$ such that any sequence of elements of $G$ of length $m$ contains a product-one subsequence.

A generalization $d_k(G)$ of $d(G)$ is defined (for abelian groups, but the definition is valid in general) in \cite{HK} for any positive integer $k$ as follows:

\begin{definition}[\cite{HK}]\label{d_k}
Let $G$ be a finite group and $k$ be a positive integer. The $k$\textit{-th Davenport constant} $d_k(G)$ is defined to be the smallest positive integer $m$, such that, given any sequence of elements of $G$ of length $m$, one can find $k$-many disjoint product-one subsequences of it.
\end{definition}

We define some classes of non-abelian groups that we will use in this paper. The dihedral group of order $2n$ is defined as $\mathbf{D}_{2n} :=  \langle x , y \mid x^n = y^2 = 1, x^{-1}y = yx \rangle$. The dicyclic group of order $4n$ is defined as $\mathbf{Q}_{4n} := \langle x, y \mid x^n = y^2, x^{2n} = 1, x^{-1}y = yx \rangle$. By abuse of notation, we denote by $\Z_a\ltimes\Z_b$ the metacyclic group $G_{a,b,g}:= \Z_{a}\ltimes_g \Z_{b} = \left\langle x, y \mid x^b = y^a = 1, x^g y = yx \right\rangle,$ with $\ord_{b}(g)=a$. One important class of metacyclic groups is that of holomorphs $\hol(\mathbb{Z}_{p^n})$ of cyclic groups of $p$-power order for an odd prime $p$; these are the semidirect product groups $Aut(\mathbb{Z}_{p^n}) \ltimes \mathbb{Z}_{p^n}\cong \Z_{p^{n-1}(p-1)}\ltimes \Z_{p^{n}} $.

\subsection{Main results}\label{results} 
We introduce two new variants of the Davenport constant for a finite group $G$.
\begin{definition}\label{f and e}
Let $G$ be a finite group with Davenport's constant $d$. We define the constant $f(G)$ as the smallest positive integer $m$ such that given any sequence $S$ of elements of $G$ of length $d$, one can always find a product-one subsequence $T \subset S$ of length at most $m$.

Similarly, we also define the constant $e(G)$ to be the smallest number $m$ such that given any product-one sequence $S$ of elements of $G$, one can always find a product-one subsequence $T$ of length at most $m$.
\end{definition}
Clearly, we have $d(G) \geq e(G) \geq f(G) \geq \max \{\ord(g):g \in G\}$. It turns out that for abelian $G$, we have $d(G)=e(G)=f(G)$. To see this, let $S'=(a_{1}, a_{2}, \ldots, a_{d(G)-1})$ be a maximal zero-sum free sequence. Define $a=a_{1}+a_{2}+\cdots+ a_{d(G)-1}$. Then $S=S'\cdot (-a)$ is a sequence of length $d(G)$ that has no proper zero-sum subsequence. Therefore, we get $f(G) \geq d(G)$ and hence we are done by the above inequalities.

As a consequence of our results, we get examples of group $G$ where $d(G)$, $e(G),$ and $f(G)$ \textbf{are all distinct.}

Our purpose is to obtain values and bounds for $d(G)$, $e(G)$, $f(G)$ and $d_k(G)$, for some classes of non-abelian (mostly metacyclic) groups $G$.\\

At first, we compute the values of $e(G)$ and $f(G)$ where $G$ is a dihedral or dicyclic group.

\begin{thm}\label{dihdic} \ 
\begin{enumerate}[label=(\alph*)]
    \item $f(\mathbf{D}_{2n})= e(\mathbf{D}_{2n})=d(\mathbf{D}_{2n})-1=n$, for any positive integer $n\ge 3$.
    \item $f(\mathbf{Q}_{4n})= e(\mathbf{Q}_{4n})=d(\mathbf{Q}_{4n})-1=2n$, for any positive integer $n\ge 2$.
\end{enumerate}
\end{thm}

Our next result is for a specific class of metacyclic groups.
\begin{thm}\label{fconstgeneraln}
Given any prime $p$ and positive integer $n>1$, we have
\begin{align*}
d(\Z_p \ltimes \Z_n) =& n+p-1;\\
f(\Z_p \ltimes \Z_n) =&n.
\end{align*}
\end{thm}

For a general metacyclic group, we conjecture the following:
\begin{conjecture}\label{conj5}
Given any positive integers $m,n$, we have $f(\Z_m\ltimes\Z_n)=n$.
\end{conjecture}

In particular, our conjecture would imply that $f(\hol(\Z_p))= p$ for all odd primes $p$. We verify this statement for $p=3, 5$. For higher value of $p$, we obtain the weaker bound:
\begin{thm}\label{holozp}
For primes $p>5$, we have  $e(\hol(\mathbb{Z}_p)) \leq p +
\frac{p-3}{2}$, and, therefore, $f(\hol(\mathbb{Z}_p)) \leq p +
\frac{p-3}{2}$.
\end{thm}

We also find a lower bound for $e(\hol(\Z_p))$ as follows:
\begin{prop}\label{ebound}
    For any prime $p>5$, we have $e(\hol(\Z_p)) \ge p+1$. Also, $e(\hol(\Z_5))=6$.
\end{prop}

This allows us to conclude that for $G=\hol(\Z_5)$, we have $f(G)=5$, $e(G)=6$, and $d(G)=8$; so they are all distinct.

In the course of our proofs, we obtain the following result, which provides an upper bound for the Davenport constant of a group $G$ in terms of a normal subgroup $H$ and the quotient $G/H$.

\begin{prop}\label{extension}
Given a short exact sequence of finite groups 
\[
1 \rightarrow H \rightarrow G \stackrel{\phi}{\rightarrow} K
\rightarrow 1,
\]
one has, $d(G) \leq (d(H)-1)f(K) + d(K)$. Further, if equality holds, then $f(G) \leq f(H) f(K)$.
\end{prop}

We also establish a bound on the $k$-th Davenport constant of $G$ in terms of $f(G)$ and $d(G)$.
\begin{prop}\label{dkintermsoff}
For any finite group $G$ and positive integer $k$, we have $d_k(G) \leq (k-1) f(G) + d(G)$.
\end{prop}

Using this proposition, we compute the $k$-th Davenport constant for several nonabelian groups. In particular, we show that $d_k(\mathbf{D}_{2n}) = nk + 1$, $d_k(\mathbf{Q}_{4n}) = 2nk + 1$ and $d_k(\Z_p\ltimes\Z_n) = nk + p -1$.

\subsection{Outline of the paper}\label{outline}
To describe the contents of the paper, we start by evaluating $f(G)$ and $e(G)$ for dihedral and dicyclic groups in \secref{f-const dih dicy gp}. 
In \secref{f-const pq group sec}, we establish several useful lemmas and prove \thmref{fconstgeneraln} in the special case where $n$ is also a prime. In Section \ref{f, d-const pn gp sec}, we evaluate$f(G)$ and $d(G)$ for $G = \mathbb{Z}_{p} \ltimes \mathbb{Z}_{n}$ with prime $p$ and general $n$ (which may or may not be divisible by $p$). This improves the result of Qu and Li (\cite{QL23}) by removing the assumption that $p$ is smaller than every prime dividing $n$. In Section \ref{hol bound sec}, we obtain the upper bound for $e(\hol(\mathbb{Z}_p))$ (and therefore, for $f(\hol(\mathbb{Z}_p))$) for primes $p>5$. We also prove in that section the equality $f(\hol(\mathbb{Z}_p))=p$ for the primes $3,5$. In the last section, using $f(G)$, we obtain potentially better bounds for $d_k(G)$ than what is known. This allows us to compute $d_k(G)$ for the groups $\mathbf{D}_{2n}, \mathbf{Q}_{4n}$ and $\Z_p \ltimes\Z_n$, where $p$ is a prime.

\vskip 5mm

\section{Computing $f(\mathbf{D}_{2n})$ and $f(\mathbf{Q}_{4n})$}\label{f-const dih dicy gp}

 In this section, we prove \thmref{dihdic}, i.e., we show that if $G$
is the dihedral group $\mathbf{D}_{2n} = \langle x ,y : x^n=y^2=1,
yxy^{-1}=x^{-1} \rangle $ of order $2n$, or the dicyclic group $\mathbf{Q}_{4n} = \langle x,y : x^{2n}=1, x^n=y^2, yxy^{-1}=x^{-1} \rangle$ of order $4n$, then $f(G)=e(G)=d(G)-1$. In both cases, we denote $H = \langle x \rangle$ and note that $\ord(x) = d(G)-1$.

\begin{proof}[\textbf{Proof of \thmref{dihdic}}]
As $\ord(x)=d(G)-1$ in both cases, we have
\begin{align*}
d(G) \geq e(G) \geq f(G) \geq d(G)-1.
\end{align*}
Hence, it is enough to show $e(G)\neq d(G)$. Suppose this is not true;
then there is a minimal product-one sequence, say $S$, of length
$d(G)$. Since $S$ is minimal, deleting any element $S$ creates a
$d(G)-1$ length sequence which is product-one-free, i.e., an
\textit{extremal product-one-free} sequence in $G$.

\vspace{2mm} We discuss \textit{(a)} first where $G=\mathbf{D}_{2n}$. In
this case $|S|=n+1$. If all $n+1$ elements are from $G \setminus H$, then some $x^{i}y$ is present at least twice. Then, we are done because $x^{i}y\cdot x^{i}y=1$. In the contrary case, there exists some $x^{a} \in S \cap H$. Thus, $S \setminus (x^{a})$ is an extremal product-one-free sequence of $\mathbf{D}_{2n}$. We consider
two cases:

\vspace{2mm}
\noindent \textit{Case I : ($n >3$)}\\
By \cite[Theorem~1.3]{MR18}, $S \setminus (x^{a})$ must then consist  of $n-1$ copies of $x^{t}$ for some $t$ coprime to $n$, and one copy of $x^{s}y$. So, $S=(x^{a}, x^{s}y, x^{t}, \ldots, x^{t})$. But then $\Pi(S) \cap H =\emptyset$, since any order of multiplying the elements will produce an element of the form $x^{c}y$; this is a contradiction to the assumption.

\vspace{2 mm}
\noindent \textit{Case II : ($n=3$)}\\
Once again, by \cite[Theorem~1.3]{MR18}, $S \setminus (x^{a})$ must be of the form $(x^{t}, x^{t}, x^{u}y)$ or $(y, xy, x^{2}y)$. In both cases, we observe that $\Pi(S) \cap H =\emptyset$ as each way of multiplying produces a term of the form $x^cy$, which leads to a contradiction. This completes the proof of the group $\mathbf{D}_{2n}$.

\vspace{2mm} Now, we consider the dicyclic group $G=\mathbf{Q}_{4n}$ of \textit{Case (b)}. Let $h$ be an element of $S$. Once again, we have two cases according to whether $n>2$ or not.

\vspace{2mm}
\noindent \textit{Case I : ($n >2$)}\\
By \cite[Theorem~1.4]{MR18}, one obtains $S \setminus (h)= (x^{s}y, x^{t}, \ldots, x^{t})$, for some integers $s$ and $t$ such that $\gcd(t, 2n)=1$ and $x^t$ appears $2n-1$ times. Now, if $h=x^{a}$ for some $a$, then $S \setminus (x^{s}y)$ cannot be extremal
product-one-free, leading to a contradiction. Also, if $h=x^{b}y$
for some $b$, then $S \setminus (x^{t})$ cannot be extremal
product-one-free. Therefore, in both cases, we obtain a
contradiction.

\vspace{2 mm}
\noindent \textit{Case II : ($n =2$)}\\
Again, by \cite[Theorem~1.4]{MR18}, the extremal product-one-free sequences $S \setminus
(h)$ of $G$ is of the form $(x^{t}, x^{t}, x^{t}, x^{s}y)$ or $(x^{t}, x^{s}y, x^{s}y, x^{s}y)$ or $(x^{s}y, x^{s}y, x^{s}y, x^{s+t}y)$, where $2 \nmid t$. It is not difficult to check that adding a fifth element to any of these sequences will not produce a minimal product-one sequence, leading again to a contradiction. This completes the proof for $\mathbf{Q}_{4n}$ also.
\end{proof}

\vspace{5 mm}

\section{Determining $f(\mathbb{Z}_{p} \ltimes \mathbb{Z}_{q})$ for primes $p,q$} \label{f-const pq group sec}

In this section, we prove \thmref{fconstgeneraln} in the special case where $n$ is a prime as well. We state this result as a proposition.
\begin{prop}\label{f-const pq group prop}
Let $p, q$ be odd primes. Then, we have $f(\mathbb{Z}_{p} \ltimes \mathbb{Z}_{q})=q$.
\end{prop}

Before beginning with the proof, we recall some useful facts from Bass (\cite{JB07}) which we state in the form we need, as two lemmata. As they are extracted from within his paper, we also give quick proofs. Bass considers the groups
\[
G= \Z_n\ltimes\Z_q = \left\langle x, y\mid x^q
= y^n = 1, x^g y = yx\right\rangle,
\]
where $q$ is a prime and ${\rm
ord} _q(g)=n$. Let $H = \langle x\rangle\cong \Z_q$ and $H^* =
H\setminus\{1\}$. The following statement has been extracted from \cite[Lemma~14]{JB07}.
\begin{lemma}{\cite{JB07}}\label{bass1}
    Let $S$ be a nonempty sequence of elements of $G$ which is minimal with respect to the property that $\Pi(S)\cap H \neq \emptyset$. Then, $\Pi(S)\subset H$ and either $1 \in \Pi(S)$ or $|\Pi(S)| \ge |S|$.
\end{lemma}

\begin{proof}
    Assume that $1\notin \Pi(S)$. Note that, since $G/H \cong \Z_n$ is abelian, either $\Pi(S)\subset H$ or $\Pi(S)\cap H = \emptyset$.

    Now, let $|S| = t$ and $S = (s_1, s_2, \ldots , s_t)$. Define the products
    \[\pi_i(S) = s_{i+1}s_{i+2}\cdots s_t s_1\cdots s_{i-1}s_i,\ \forall \ 0\le i\le t-1.\]
    Let $\pi_i(S) = \pi_j(S)$ where $i<j$. Then, let $s_{i+1}\cdots s_{j} = h_1$ and $s_{j+1}\cdots s_{t}s_{1}\cdots s_{i} = h_2$.

    By our assumption $h_1 h_2 = h_2 h_1 \in H^{*}$. Therefore, by the next lemma, we will have $h_1 = s_{i+1}\cdots s_{j} \in H$, which contradicts the minimality of $S$. Therefore, $\pi_i(S)$ are all distinct. Since there are $t$ many of them, and they are all elements of $\Pi(S)$, we get that $|\Pi(S)|\ge t = |S|$.
\end{proof}

\begin{lemma}\label{bass2}
    Let $h_1 = x^a y^b$ and $h_2 = x^c y^d$ such that $h_1 h_2 = x^k \in H^{*}$. Then, $h_2 h_1 = x^{g^d k}$. In particular, if $h_1 h_2 = h_2 h_1 \in H^{*}$, then $h_1, h_2 \in H$.
\end{lemma}

\begin{proof}
    Since $h_1 h_2 = x^k \implies x^k = x^a y^b x^c y^d = x^{a+g^b c} y^{b+d}$. Therefore, we get that $b+d \equiv 0 \pmod{n}$ and $k \equiv a + g^{b} c \pmod{q}$.

    Now, $h_2 h_1 = x^c y^d x^a y^b = x^{c+g^d a}y^{b+d} = x^{c+g^d a}$. Since, $b+d \equiv 0 \pmod{n}$ and ${\rm ord} _q(g) = n$, we get
    \[c+g^d a \equiv g^{d}(a + g^{-d}c) \equiv g^{d}(a+g^{b} c) \equiv g^{d} k \pmod{q}. \]
    Therefore, $h_2 h_1 = x^{g^{d}k}$.

    Now, if $h_1 h_2 = h_2 h_1 = x^k$, we will have $g^d k \equiv k \pmod{q}\implies q\mid k(g^d - 1)$.  Since $x^k \in H^{*}$ and ${\rm ord} _q(g) = n$, we get $n\mid d\implies h_2 \in H\implies h_1 = (h_2^{-1})(h_2 h_1) \in H$.
\end{proof}

Before stating the other lemma extracted from Bass's paper, we need to recall the Cauchy-Davenport inequality which is a very helpful tool while studying zero-sum problems in abelian groups.
\begin{lemma} \cite[Theorem~6.2]{G13}\label{CD} \textbf{Cauchy-Davenport Inequality}.
Let $p$ be a prime and $A_{1}, A_{2}, \ldots, A_{\ell}$ be nonempty subsets of $\mathbb{Z}_{p}$. Then
\[
|A_{1}+A_{2}+ \cdots+A_{\ell}| \geq \min \left\{p, \sum_{i=1}^{\ell}|A_{i}|-\ell+1\right\}
\]
where $A_{1}+A_{2}+\cdots+A_{\ell}\coloneq \{a_{1}a_{2} \cdots a_{\ell} : a_{i} \in A_{i}$ for $i=1, 2, \ldots, \ell \}$.
\end{lemma}

Several proofs in this paper utilize the following lemma which can, in principle, be crystallized from \cite[Lemma~14]{JB07}, although it does not appear explicitly. So, we state the result and prove it as well.

\begin{lemma}{\cite{JB07}}\label{usefulbass}
    Let $G$ and $H$ be as in \lemref{bass1}. 
    \begin{enumerate}[label=(\alph*)]
        \item If $T$ is a sequence of elements of $G$ such that $|T| \geq q$ and $\Pi(T) \cap H \neq \emptyset$, then $T$ contains a product-one subsequence.
        \item Let $T_1, T_2, \ldots, T_k$ be nonempty sequences of elements of $G$ such that $\Pi(T_i) \subseteq H\,\forall\,i$, and $\displaystyle\sum_{i=1}^{k}|\Pi(T_i)|\ge q$. Let $T = T_1T_2 \cdots T_k$. Then $T$ contains a product-one subsequence.
    \end{enumerate}
\end{lemma}

\begin{proof}
We start by proving part $(b)$. If $1=x^0  \in \Pi(T_i)$ for some $i$, then we are done. If not, we define $T_1' = \Pi(T_1)$, and $T_i' = \Pi(T_i) \cup \{1\}$ for $2  \leq i  \leq k$. Write $\Sigma = \sum_{i=1}^k T_i'$. Observe that each element of $\Sigma$ can be obtained as a product of some subsequence of $T$. By the Cauchy-Davenport inequality, we get 
    \begin{align*}
    |\Sigma| \geq& \min \left\{q, \sum_{i=1}^k |T_i'|-k+1 \right\}\\ =& \min \left\{q, \left(\sum_{i=1}^k |\Pi(T_i)|\right) + k-1-k+1 \right\} \geq  \min \left\{q, \sum_{i=1}^k |\Pi(T_i)| \right\} \geq q.
    \end{align*}
    Therefore, $1 \in \Sigma$, and we are done proving part $(b)$.

    For part $(a)$, note that $G/H \cong \mathbb{Z}_n$ which is abelian; so $\Pi(T) \subseteq H$. Now, we partition $T$ into subsequences $T_1, T_2, \ldots, T_k$ such that each $T_i$ is minimal for the property that $\Pi(T_i) \subseteq H$. Then $\sum_{i=1}^k |T_i|=|T| \geq q$. If $1\in \Pi(T_i)$ for some $i$, then we are done. Otherwise, by \lemref{bass1}, we have $|\Pi(T_i)| \geq |T_i|\,\forall\,1 \leq i \leq k$. Therefore, $\sum_{i=1}^{k}|\Pi(T_i)|\ge q$, and we are done by part $(b)$.

\end{proof}

\begin{proof}[\textbf{Proof of \propref{f-const pq group prop}}]
Let $G= \mathbb{Z}_{p} \ltimes \mathbb{Z}_{q}= \left\langle x, y : x^{q}=y^{p}=1, x^{g}y= yx \right\rangle$, where $\ord_{q}(g)=p$. Let $H=\langle x\rangle \cong \mathbb{Z}_{q}$.
Since ${\rm ord} (x)=q$, we have $f(G) \geq q$. Now, given a
sequence $S$ of length $d(G)=p+q-1$, we must find a product-one
subsequence of length less than or equal to $q$. Let $S=
(g_{1},g_{2}, \ldots, g_{p+q-1})$, where $g_{i}=x^{a_{i}} y^{b_{i}}\,\forall\,i$. Define $S_{y}=(b_{1}, b_{2}, \ldots, b_{p+q-1})$, a sequence in $\mathbb{Z}_{p}$.

\vspace{2mm}
\noindent
\textit{Case I : (No element of $\mathbb{Z}_{p}$ is present in $S_{y}$ more than $q$ times.)}\\
In this case, $S_{y}$ can be written as a concatenation of disjoint subsequences $B_{1}, B_{2}, \ldots, B_{q}$ where no $B_i$ has any repeating elements as follows:

\vspace{1mm} We divide elements of $S_{y}$ into $q$
sets. At first, we take the first $q$ elements of $S_{y}$
and create $q$ singleton sets. Then we take the remaining elements
of $S_{y}$ one by one and place them in one of the sets that does not already contain a copy of it. We can keep doing this because no
element of $\mathbb{Z}_{p}$ is present in $S_{y}$ more than $q$
times, so at no point will we pick an element which has a copy of it
present in each of the $q$ sets we are building. The sets we build this way will correspond to $B_1, B_2, \ldots, B_q$. Therefore, we have
\[
\sum_{i=1}^{q} |B_{i}|=|S_{y}|= p+q-1.
\]

Thus, by using Cauchy-Davenport inequality, we obtain
\[
\left|\sum_{i=1}^{q} B_{i}\right| \geq \min\left\{p, \sum_{i=1}^{q} |B_{i}|-q+1\right\}=p.
\]
This implies that $0 \in \sum_{i=1}^{q} B_{i}$, that is there are
$q$ elements $c_{1}, c_{2}, \ldots, c_{q}$ in $S_{y}$ such that
$\sum_{i=1}^{q}c_{i}=0$ in $\mathbb{Z}_{p}$. Let $h_{1}, h_{2},
\ldots, h_{q}$ be the corresponding elements in $G$. Let $S'=(h_{1},
h_{2}, \ldots, h_{q})$. Then $\Pi(S')\subseteq H$.  Since $|S'|= q$, we get by \lemref{usefulbass}$(a)$ that $|S'|$ contains a product-one subsequence. Therefore, we are done in this case.

\vspace{2mm}
\noindent
\textit{Case II : (Some element $b$ of $\mathbb{Z}_{p}$ is present in $S_{y}$ more than $q$ times.)}\\
If $b=0$, we have $|S \cap H| > q$. Since $D(\Z_{q}) = f(\Z_q) =q$, we will get a product-one subsequence of $S \cap H$ of length less than or equal to $q$. So, we are done in this case.

If $b \neq 0$, let $g\equiv (g')^b \pmod{q}$. Since there is an isomorphism $G\xrightarrow{\sim} \Z_p\ltimes_{g'}\Z_q$ that sends $x\mapsto x$, $y\mapsto y^b$, we can assume, without loss of generality, that $b=1$. Then, we have a subsequence of $S$ of the form $S'=(x^{a_{1}}y, x^{a_{2}}y, \ldots, x^{a_{q+1}}y)$. Since $p \mid q-1$, we notice that $q+1 \geq 2p+2$. If $a_{i}$ consists of only one or two distinct values, then one of them must be present at least $p$ times. Since, we have
\begin{align*}
(x^{a}y)^{p}= x^{a+ga+ \cdots+ g^{p-1}a} y^{p}= x^{\frac{a}{g-1}(g^{p}-1)}=1,
\end{align*}
we are done in this case.

Now, we assume that there exists a subsequence $S_{1}= (x^{b_{1}}y, x^{b_{2}}y, x^{b_{3}}y)
\subset S^{'}$ consisting of $3$ distinct elements. Clearly, we have $|S' \setminus S_{1}| \geq 2p-1$. So, if $S' \setminus S_{1}$ was
made of only one or two distinct elements, one of them is present at
least $p$ times, and we are done again. Therefore, we may assume
that there is a subsequence $S_{2}= (x^{c_{1}}y,
x^{c_{2}}y, x^{c_{3}}y) \subset S^{'} \setminus S_{1}$ which consists of $3$ distinct elements.

\vspace{2mm}

For continuing with the proof, we need two lemmata which
we state and prove; they are of independent interest.

\begin{lemma} \label{3 dist elts lem}
Let $a_{1},a_{2},a_{3}$ be three distinct elements in
$\mathbb{Z}_{q}$ and $u$ be a positive integer. Define the set
\[
A= \left\{a_{\sigma(1)}+ua_{\sigma(2)}+u^{2}a_{\sigma(3)} : \sigma \in S_{3}\right\}.
\]
\begin{enumerate}[label=(\alph*)]
\item If $\ord_{q}(u) > 3$, then $|A| \geq 4$.
\item If $\ord_{q}(u)=3$, then $0 \in A$ or $|A|=6$.
\end{enumerate}
\end{lemma}
\begin{proof}
The set $A$ consists of the following six elements:
\begin{align*}
&h_{1}= a_{1}+a_{2}u+a_{3}u^{2}, h_{2}= a_{2}+a_{3}u+a_{1}u^{2}, h_{3}=a_{3}+a_{1}u+a_{2}u^{2}\\
&h_{4}= a_{3}+a_{2}u+a_{1}u^{2}, h_{5}= a_{2}+a_{1}u+a_{3}u^{2},
h_{6}=a_{1}+a_{3}u+a_{2}u^{2}.
\end{align*}
Let $T_{1}= \{h_{1}, h_{2}, h_{3}\}$, $T_{2}=\{h_{4}, h_{5},
h_{6}\}$. For $i \in \{1,2,3 \}$ and $j \in \{4,5,6 \}$, notice that
\[
h_{i}-h_{j}= (a_{r}-a_{s})\left(u^{v}-u^{w}\right), \ \text{where} \ r \neq s \in \{1,2,3\} \ \text{and } v \neq w \in \{0,1,2\}.
\]
This implies $h_{i} \neq h_{j}$ as $a_{i}$ are distinct and ${\rm ord} _{q}(u) >2$. Therefore, we have $T_{1} \cap T_{2} = \emptyset$.\\

\noindent{Case I : (${\rm ord} _{q}(u)=3$)}\\
Let us assume that $ 0 \notin A$. Then, $h_{1}= uh_{2}= u^{2}h_{3}$, which gives us $h_{1} \neq h_{2}, h_{2} \neq h_{3}, h_{3} \neq h_{1}$, and
hence we have $|T_{1}|=3$. Similarly, we also have $|T_{2}|=3$. From
this we conclude that $|A|=6$.\\

\noindent \textit{Case II : (${\rm ord} _{q}(u) > 3$)}\\
Suppose that $|T_{1}|=1,$ that is $h_{1}=h_{2}=h_{3}$. Then, $h_{1}=h_{2}$ implies that
\begin{align*}
&a_{1}+a_{2}u+a_{3}u^{2}= a_{2}+a_{3}u+a_{1}u^{2} \implies (a_{1}-a_{2})+u(a_{2}-a_{3})=(a_{3}-a_{1}) (-u^{2})\\
& \implies (a_{1}-a_{2})(1-u)=(a_{3}-a_{1})(u-u^{2}) \implies (a_{1}-a_{2})=u(a_{3}-a_{1}).
\end{align*}
By symmetry, $h_{2}=h_{3}$ implies that $(a_{2}-a_{3})=u (a_{1}-a_{2})= u^{2}(a_{3}-a_{1})$ and $h_{3}=h_{1}$ implies that $(a_{3}-a_{1})=u(a_{2}-a_{3})=u^{3}(a_{3}-a_{1})$. This gives us $u^{3}=1$ which is a contradiction to the assumption $\ord_{q}(u) >3$. Therefore, $|T_{1}| \geq 2$, and by symmetry, $|T_{2}| \geq 2$ as well. Thus, $|A|=|T_{1}|+|T_{2}| \geq 4$. This completes the proof of \lemref{3 dist elts lem}.
\end{proof}

\begin{lemma}\label{inductive lem}
Let $A \subseteq \mathbb{Z}_{q}$ and $a, u \in \mathbb{Z}_{q}$ such that $1 < |A| < {\rm ord} _{q}(u)$. Let $c \in \mathbb{N}$. Define the sets $A_{1}= \{a+ub : b \in A\}$ and $A_{2}=\{b+u^{c}a: b \in A\}$. Then $|A_{1} \cup A_{2}| \geq |A|+1$.
\end{lemma}
\begin{proof}
Let $n=|A|$, and $A=\{a_{1}, a_{2}, \ldots, a_{n}\}$. Clearly,
$|A_{1}|=|A_{2}|=n$. So, if $|A_{1} \cup A_{2}| \leq |A|$, then
$A_{1}=A_{2}$. Thus, the elements of $A_{1}$ must be in one-to-one
correspondence with the elements of $A_{2}$, that is, there exists
$\sigma \in S_{n}$ such that $a+ua_{i}= a_{\sigma(i)}+u^{c}a$. Now,
break $\sigma$ into disjoint cycles. We observe that if $a+ua_{i}=
a_{i}+u^{c}a$, then $a_{i}=a \left(\tfrac{u^{c}-1}{u-1}\right)$. Since the
$a_{i}$'s are distinct, at most one $1$-cycle can be there.

Take a cycle of length $k > 1$. Without loss of
generality, we can assume that the cycle is $(1, 2, \ldots, k)$.
For the rest of the proof, we will take the indices $i$ of $a_i$ modulo $k$.

We have $a+ua_{i}=a_{i+1}+u^{c}a\,\forall\, 1 \leq i \leq k$. Equivalently, we get $ua_{i} - a_{i+1}=a (u^{c}-1)\, \forall\, 1 \leq i \leq k$. Consequently, for all $1 \leq i \leq k$, we obtain $u a_{i}-a_{i+1}= u a_{i+1}-a_{i+2} \implies u(a_{i}-a_{i+1})= a_{i+1}-a_{i+2}$. Multiplying the left-hand side and
right-hand side of these $k$ equalities, we get: 
\[
u^{k} \prod_{i=1}^{k} (a_{i}-a_{i+1})=  \prod_{i=1}^{k} (a_{i}-a_{i+1}).
\]
However, this simplifies to $u^{k}=1$, which contradicts the
assumption $k \leq n =|A| < {\rm ord} _{q}(u)$.
\end{proof}

Returning to the proof of the theorem, we consider three subcases:\\

\noindent\textit{Subcase I: } We assume $p >5$ for now. Let
\[
S'\setminus S_{1} S_{2}= \left(x^{c_{4}}y, x^{c_{5}}y, \ldots,
x^{c_{p-3}}y\right)\cdot\left(x^{d_{p+1}}y, x^{d_{p+2}}y, \ldots, x^{d_{q+1}}y\right).\]
Let $W_{1}=S_1$, $W_{3}= S_2$ and for all $3 \leq i
\leq p-4$, let $W_{i+1}=W_{i} \cdot (x^{c_{i+1}}y)$. Since $c_{1},
c_{2}$ and $c_{3}$ are distinct, by \lemref{3 dist elts lem} (applied to the exponents of $x$), we have $|\Pi(W_{3})| \geq 4$. Since every element of $\Pi(W_{i})$ looks like $x^{d}y^{i}$, we observe that if $\Pi(W_{i}) \supseteq
\{x^{d_{j}}y^{i}\}_{j=1}^{i+1}$, then 
\[
\Pi(W_{i+1}) \supseteq
\left\{x^{d_{j}+g^{i}c_{i+1}} y^{i+1}\right\}_{j=1}^{i+1} \cup
\left\{x^{c_{i+1}+gd_{j}} y^{i+1}\right\}_{j=1}^{i+1}.
\]
Then, by applying \lemref{inductive lem} to the exponents of $x$ and using induction on $i$, we get $|\Pi(W_{i})| \geq i+1$ for all $3 \leq i \leq p-3$. In particular, $|\Pi(W_{p-3})|
\geq p-2$. Also, by \lemref{3 dist elts lem}, $|\Pi(W_{1})| \geq
4$.

Let $X_{1}= W_{1} W_{p-3}$. Then
\[
\Pi(X_{1}) \supseteq \left\{x^{a+g^{3}b}: x^{a}y^{3} \in \Pi(W_{1}), \ x^{b}y^{p-3}\in \Pi(W_{p-3})\right\}.
\]
For any subsets $A, B$ of $\mathbb{Z}_{q}$ with $|A|=4$, $|B|=p-2$, it follows from Cauchy-Davenport inequality that
\[
|A+g^{3}B| \geq \min\{q, |A|+|B|-1 \}=\min\{q, p+1\}.
\]
Thus, we have $|\Pi(X_{1})| \geq p+1$. Let $\frac{q-1}{p}=r$.
Define $X_{i}=(x^{d_{(i-1)p+1}}y, \ldots, x^{d_{ip}}y)$ for every $2 \leq i \leq r$. Note that $\Pi(X_{i}) \subseteq H$ for all $1\le i\le r$. If $1\in \Pi(X_i)$ for some $i>1$, then we are done. Otherwise, by \lemref{bass1}, $|\Pi(X_i)|\ge |X_i|= p \,\forall\,i>1$. Therefore, 
\[
\sum_{i=1}^{r}|\Pi(X_i)| \ge (p+1) + (r-1)p = rp+1 = q.
\]

Now, by \lemref{usefulbass}$(b)$, the sequence $X\coloneq X_{1} X_{2}\cdots X_{r}$ contains a product-one subsequence. Since, $|X| = rp = q-1$, we are done.

\vspace{2mm}
\noindent
\textit{Subcase II: ($p=3$ case)}\\
Again let  $W_{1}=(x^{b_{1}}y, x^{b_{2}}y, x^{b_{3}}y)$,
$S \setminus W_{1}=(x^{d_{4}}y, \ldots, x^{d_{q+1}}y)$ and
$W_{i}=(x^{d_{3i-2}}y, x^{d_{3i-1}}y, x^{d_{3i}}y)$ for all $2 \leq
i \leq r$, where $r=\frac{q-1}{3}$. Note that $\Pi(W_i)\subseteq H$ for all $1\le i\le r$. If $1\in\Pi(W_i)$ for any $i$, then we are done. Otherwise, by \lemref{3 dist elts lem}, we have $|\Pi(W_{1})| \geq 6$, and \lemref{bass1} implies that $|\Pi(W_{i})| \geq 3 \,\forall\,i>1$. Therefore,
\[\sum_{i=1}^{r}|\Pi(W_i)| \ge 6 + (r-1)3 = 3r+3 = q+2.\]
Thus, \lemref{usefulbass}$(b)$ implies that $W\coloneq W_1 W_2 \cdots W_r$ contains a product-one subsequence. Since, $|W|= 3r = q-1$, we are done in this subcase.

\vspace{2mm}
\noindent
\textit{Subcase III:  ($p=5$ case)}\\
Recall that $S'=(x^{a_{1}}y, x^{a_{2}}y, \ldots, x^{a_{q+1}}y)
\subset S$. We show that, either we can construct disjoint $3$ element subsequences $W_{1}, W_{2}, W_{3}$  of $S'$ such that each $W_i$ consists of $3$ distinct elements, or we are done (in the sense that we get a product-one subsequence of $S'$ of length $\le q$). 

Suppose some $a_{i}$ appears more than $4$ times; then we are done, since $(x^{a_i}y)^5=1$. Therefore, we assume that no $a_{i}$ is present more than 4 times. Since $q+1 \geq 12$, there must be 3 distinct elements among $a_{i}$s. If no $a_{i}$ is present more than twice, then by pigeonhole principle, we can find $W_{1} \subset S'$, $W_{2} \subset S'\setminus W_{1}$ and $W_{3} \subset S'\setminus (W_{1} W_{2})$, where $|W_{i}|=3$ consists of 3 distinct elements each. If not, then some number $b$ is present $3$ or $4$ times among $a_{i}$. Focus on the elements different from $x^b y$. If none of them are present more than twice, again by pigeonhole principle, we can find disjoint (as subsequences) $V_{1}, V_{2}, V_{3}$ from them such that $|V_{i}|=2$ consists of 2 distinct elements. Then, we can set $W_{i}= V_{i}\cdot(x^{b}y)$. If even the last assumption does not hold, then some $c \neq b$ is also present 3 or 4 times. Therefore, we get disjoint $W_{1},W_{2},W_{3}$ by pairing $(x^b y, x^c y)$ with any remaining terms. So, we are done constructing $W_1, W_2, W_3$ in all possible cases.

Let $q-1=5r$ and let $S'\setminus W_{1} W_{2} W_{3}= (x^{m} y, x^{d_{11}}y, x^{d_{12}} y, \ldots, x^{d_{q+1}}y)$. Let $W_{4}=W_{3}\cdot (x^{m}y)$, $V_{1}= W_{1} W_{2} W_{4}$ and $V_{i}=(x^{d_{5i+1}}y, \ldots, x^{d_{5(i+1)}}y)$ for all $2 \leq i \leq r-1$. Note that $\Pi(V_i)\subseteq H$ for all $1\le i\le r$. If $1\in \Pi(V_i)$ for any $i$, then we are done. Otherwise, on using the \lemref{bass1}, we get $|\Pi(V_{i})| \geq 5$ for all $2 \leq i \leq r-1$. From \lemref{3 dist elts lem}, we obtain $|\Pi(W_{i})| \geq 4$ for all $1 \leq i \leq 3$ and we apply \lemref{inductive lem} to the exponents of $x$ to obtain $|\Pi(W_{4})| \geq 5$. Since \[
\Pi(V_1)\supseteq \left\{x^{a+ g^3 b + g^6 c}: x^a y^3 \in \Pi(W_1), x^b y^3 \in \Pi(W_2), x^c y^4 \in \Pi(W_4)\right\},
\]
we get by Cauchy-Davenport inequality $|\Pi(V_{1})| \geq 4+4+5-2 =11$, and hence
\[\sum_{i=1}^{r} |\Pi(V_{i})| \geq 11+ 5(r-2)=5r+1=q.\]

Now, by \lemref{usefulbass}$(b)$, the sequence $V\coloneq V_1 V_2 \cdots V_{r-1}$ has a product-one subsequence and $|V| = 10 + 5(r-2) = q-1$. This completes the proof of \propref{f-const pq group prop}.
\end{proof}

\vskip 5mm

\section{Computing $d(G)$ and $f(G)$ for $\mathbb{Z}_{p} \ltimes \mathbb{Z}_{n}$, $p$ prime}\label{f, d-const pn gp sec}

In this section, we prove \thmref{fconstgeneraln}. Before we discuss the proof, we develop a few preliminary results. 

We first prove \propref{extension}, which is not only useful here but also of independent interest. The Davenport constant for a group is trivially bounded above in terms of the corresponding constants for any normal subgroup and the quotient. However, using the $f$-constant, we can prove a potentially stronger result.

\begin{proof}[\textbf{Proof of \propref{extension}}]
Start with any sequence $S$ in $G$ with $|S| = (d(H)-1)f(K) + d(K)$.
Let us keep selecting disjoint product-one subsequences $T_1, \ldots, T_t$ of $\phi(S)$ of length $\leq f(K)$ such that until we are no longer able to do so. Let $S_1, \ldots, S_t$ be the corresponding subsequences in $S$. Now, $\Pi(S_i) \cap H \neq \emptyset$ for each $1\le i \leq t$. Further, from
the definition of $f(K)$, we have
\[
|\phi(S)\setminus T_1 \cdots T_t|\le d(K) \implies |S \setminus S_1 \cdots S_t| < d(K).
\]
Hence, $d(K) > |S| - \sum_{i=1}^t |S_i| \geq (d(H)-1)f(K)+d(K) - t
f(K)$, which implies $t> d(H)-1$. Thus, $t \geq d(H)$. Choose $s_i \in \Pi(S_i) \cap H$ for $1 \leq i \leq d(H)$. We have a
product-one subsequence $T' = (s_{a_1}, \ldots, s_{a_n})$ of length
at the most $f(H)$. Therefore, in $S$, we have the product-one
subsequence $T = S_{a_1} \cdots S_{a_n}$ such that
\[
|T| = \sum_{i=1}^n \left|S_{a_i}\right| = \sum_{i=1}^n \left|T_{a_i}\right|  \leq n f(K) \leq f(H) f(K).
\]
This completes the proof of both assertions.
\end{proof}

Next, we recall another very useful
tool in the combinatorial theory of finite abelian groups. It is a generalization of Kneser's theorem, which is called the DeVos-Goddyn-Mohar (DGM) theorem. Let $G$ be a finite abelian group and $\mathbf{A}=(A_{1}, \ldots, A_{\ell})$ be a sequence of finite subsets of $G$. For any $m \leq \ell$, we define
\[
\Pi^{m}(\mathbf{A})\coloneq\left\{a_{i_{1}}\cdots a_{i_{m}}: 1 \leq i_{1} < \cdots < i_{m} \leq \ell \ \text{and} \ a_{i_{j}} \in A_{i_{j}} \ \text{for every} \ 1 \leq j \leq m\right\}.
\]
For any subset $A$ of $G$, its stabilizer is given by $\Stab(A)=\{g \in G : gA = A\}$.
\begin{lemma}\cite[Theorem~13.1]{G13}\label{DGM} \textbf{DGM theorem.}
Let $\mathbf{A}=(A_{1}, \ldots, A_{\ell})$ be a sequence of finite subsets of a finite abelian group $G$. Let $m \leq \ell$, and $H=\Stab(\Pi^{m}(\mathbf{A}))$. If $\Pi^{m}(\mathbf{A})$ is nonempty, then
\[
\left|\Pi^{m}(\mathbf{A})\right| \geq |H| \left(1-m+ \sum_{Q \in G/H} \min \left\{m, \left|\{i \in [1, \ell]: A_{i} \cap Q \neq \emptyset \right\}\right| \}\right).
\]
\end{lemma}

Lastly, we treat a special case of \thmref{fconstgeneraln} separately where $n$ is a power of $p$. Our argument provides a new proof of $d(\mathbb{Z}_p \ltimes \mathbb{Z}_{p^k}) = p^{k}+p-1$ as well.

\begin{prop}\label{df-ppower}
Given any odd prime $p$ and positive integer $k>1$, we have
\[
d(\Z_p \ltimes \Z_{p^k}) = p^k+p-1, \text{ and }
f(\Z_p \ltimes \Z_{p^k}) = p^k.
\]
\end{prop}
\begin{proof}
Let $G = \mathbb{Z}_p\ltimes \mathbb{Z}_{p^k} =\langle x,y| x^{p^k}=y^p=1, x^gy=yx \rangle$, where the positive integer $g$ has order $p$ modulo $p^k$. Since the sequence
\[
(\underbrace{x, \ldots, x}, \underbrace{y, \ldots, y})
\]
where $x$ occurs $p^k -1$ times and $y$ occurs $p-1$ times is product-one free, we get $d(G) \geq p^k +p-1$. Also, $f(G)\ge \ord(x) = p^k$.

It is easy to observe that the center of $G$ is
\[
Z(G) = \left\langle x^p \right\rangle \cong \mathbb{Z}_{p^{k-1}}.
\]
We further note that $G/Z(G)$ is abelian, and congruent to the additively written group $\mathbb{Z}_p \oplus \mathbb{Z}_p$. Therefore, we have a short exact sequence
\[
1 \rightarrow Z(G) \rightarrow G \stackrel{\phi}{\rightarrow}
\mathbb{Z}_p \oplus \mathbb{Z}_p \rightarrow 1.
\]
\vspace{1 mm}

\noindent{\it Step I:} We claim that if $h \in G \setminus Z(G)$, then
$\phi(C_G(h)) \cong \mathbb{Z}_p$, where $C_G(h)$ denotes the
centralizer. To see this, we just need to note that $C_G(h)$ for $h$ as above
contains $Z(G)$ properly and is contained in $G$ properly; hence it contains $Z(G)$ as a subgroup of
index $p$.\\

\noindent {\it Step II:}  Consider any $S$ with $|S| = p^k+p-1$. Fix
a minimal subsequence $T \subset S$ such that $\Pi(\phi(T))= \{(0,0)\}$.
We claim that if $|T|>p$, then $|\Pi(T)|> 1$.

To prove this, write $T = h_1 h_2 \cdots h_t$ with $t>p$. Since $T$ is minimal, we have  $\phi(h_{1}) \neq (0,0)$ which implies that $h_1 \not \in Z(G)$. If $h_2,
\ldots, h_t$ are in $C_G(h_1)$, then $\phi(T) \subseteq
\phi(C_G(h_1)) \cong \mathbb{Z}_p$ by \textit{Step I}. Since $d(\Z_{p})=p$ and $t > p$, we get a proper subsequence of $\phi(T)$ whose sum is $(0,0)$. This contradicts the minimality of $T$. Thus, there exists $h_i \not\in C_G(h_1)$.
Hence, the products
$$h_1h_ih_2 \cdots h_{i-1}h_{i+1} \cdots h_t$$
and
$$h_ih_1h_2 \cdots h_{i-1}h_{i+1} \cdots h_t$$
are unequal; thus $|\Pi(T)|>1$ and \textit{Step II} holds.\\

To continue with the proof for $n=p^k$, consider a sequence $S$ of elements of $G$ of length $p^k + p -1$. Note that
\[
\Z_p \oplus \Z_p = \bigcup_{m=0}^{p-1} \left\langle (1,m) \right\rangle \cup \left\langle (0,1) \right\rangle.
\]
Write $C_m =  \langle (1,m) \rangle$ for $0 \leq m <p$, and $C_p = \langle (0,1) \rangle$. By the pigeon-hole principle, there is some $1\le m\le p$ such that $C_m$  contains at least $\lceil (p^k+p-1)/(p+1) \rceil = p^{k-1}$ elements of $S$. Since $\phi(C_m) \cong \Z_p$, and $d(\Z_p) =p$, we find a subsequence (say $K_1$ ) of $S$ whose length is at most $p$, and whose product is in $Z(G)$. Now, the idea is to find as many disjoint minimal subsequences of $S \setminus K_1$ as possible whose products are also in $Z(G)$. Let $K_1, K_2, \ldots, K_i$ be those minimal subsequences with product in $Z(G)$ with length at most $p$, and let $L_1, \ldots, L_j$ are those minimal subsequences of $S$ whose products are in $Z(G)$ and whose lengths are in $[p+1, 2p-1]$, noting that $d(\Z_p \oplus \Z_p) = 2p-1$\cite{Ol1}. By \textit{Step II} above, each $|\Pi(L_a)| \geq 2$.

Now, $|S \setminus K_1 \cdots K_i L_1 \cdots L_j| < 2p-1$; else, we would have found more subsequences of $S$ with product in $Z(G)$.
Thus, we have
\[
2p-1 > p^k+p-1- \sum_{r=1}^i |K_r| - \sum_{s=1}^j |L_s| \geq
p^k+p-1-ip- (2p-1)j.
\]
Simplifying, this gives $i+2j \geq p^{k-1}$.
Since $i$ is at least $1$, we may throw away some sequences, if
necessary, and assume $i+2j=p^{k-1}$. Thus, we have a subsequence $T=K_1 \cdots K_i L_1 \cdots L_j$ of $S$
such that $i+2j=p^{k-1}$, each $K_r$ and each $L_s$ has product
contained in $Z(G)$, each $K_r$ has length $\leq p$ and each $L_s$'s has length in the range $[p+1, 2p-1]$. Note that
\[
|T| = \sum_r |K_r| + \sum_s |L_s| \leq pi+(2p-1)j= p^k-j \leq
p^k.
\]
Let us note also that $\Pi(K_r), \Pi(L_s) \subset Z(G)$ for
$r \leq i, s \leq j$ because $G/Z(G)$ is abelian. 

Now, $i+j \geq (i+2j)/2 =p^{k-1}/2 > p^{k-2}$.

\vspace{1mm}
{\it As $Z(G) \cong \mathbb{Z}_{p^{k-1}}$, we will find it helpful now to use additive notation.}

If some $\Pi(K_i)$ or $\Pi(L_s)$ contain $0$, then we are done. Otherwise, we consider
\[
K'_1 = \Pi(K_1), K'_r = \Pi(K_r) \cup \{0\} (r>0), L'_s = \Pi(L_s) \cup \{0\}.
\]

Denote $\Sigma := \sum_{r=1}^i K'_r + \sum_{s=1}^j L'_s$. We look at the stabilizer of $\Sigma$ in $Z(G)$. \\

\noindent
{\it Case I: $\Stab(\Sigma) = (0)$.}\\
Then, by the DGM theorem (\lemref{DGM}), we have 
\[
|\Sigma| \geq \min\left\{p^{k-1}, \sum_{r=1}^i |K'_r| + \sum_{s=1}^j
|L'_s| -i-j+1\right\} \geq \min \left\{p^{k-1}, 2i-1+3j-i-j+1\right\} = p^{k-1}.
\]
In other words, $\Sigma = Z(G)$. This is a contradiction, since $\Stab(Z(G))= Z(G)$. So we are done.\\

\noindent  {\it Case II: $\Stab(\Sigma) \neq (0)$.}\\
In this case, $k>1$. Also, $\Stab(\Sigma) \supseteq
\langle p^{k-2} \rangle$ in $\mathbb{Z}_{p^{k-1}}$. Recalling that $i+j > p^{k-2}$ as observed earlier, if we pick $k_r \in K_r, \ell_s \in L_s$ for all $r \leq i, s \leq j$, then there is a subsequence of $(k_1, \ldots k_i, \ell_1, \ldots \ell_j)$ whose sum (in $\mathbb{Z}_{p^{k-1}}$) is divisible by $p^{k-2}$. So, there is some integer $c$ such that $c p^{k-2} \in \Sigma$. Therefore, considering $-c p^{k-2}$ which is in the stabilizer, we get $0 \in \Sigma$, and we are done as seen above in \textit{Case I}.
\end{proof}

\vspace{4 mm}
Now we prove \thmref{fconstgeneraln}.
\begin{proof}[\textbf{Proof of \thmref{fconstgeneraln}}]  As $p=2$ is already dealt with in \secref{f-const dih dicy gp}, we assume $p$ is odd. Let $G = \mathbb{Z}_p\ltimes \mathbb{Z}_{n} =\langle x,y| x^{n}=y^p=1, x^gy=yx \rangle$, where $\ord_n(g) = p$. Consideration of the sequence
\[
(\underbrace{x, \ldots, x}, \underbrace{y, \ldots, y})
\]
where $x$ occurs $n-1$ times and $y$ occurs $p-1$ times, shows that $d(G) \geq n+p-1$. Write $n = p_1^{a_1} \cdots p_{\ell}^{a_{\ell}}$. As $g \not\equiv 1 \pmod{n}$, by the Chinese remainder theorem, there exists an $i$ for which $g \not\equiv 1 \pmod{p_i^{a_i}}$. Taking $q=p_i$, we have $\ord_{q^{a_i}}(g)=p$ because the order divides $p$ while it is not equal to $1$.\\

\noindent  \textit{Case I : ($q\neq p$)}\\
Then, $p|(q-1)$.  Let $h$ be a primitive root modulo $q^{a_i}$. Then, we have $\ord_q(h) = q-1$ and $g = h^{cq^{a_i-1}(q-1)/p}$ for some integer $c$ with $(p,c)=1$. Since $q-1\nmid cq^{a_i-1}\frac{q-1}{p}$, we have $g \not\equiv 1 \pmod{q}$.

Let $n=mq$ and write $H_1\coloneq \langle x^q \rangle \cong \mathbb{Z}_m$. Note that $H_1$ is normal in $G$. Further, we observe that $G/H_1 \cong \mathbb{Z}_p \ltimes \mathbb{Z}_q$. For, if not, the images of $x,y$ in $G/H_1$ would commute, which would imply that $xyH_1=x^gyH_1$ and hence, $x^{g-1} \in H_1= \langle x^q \rangle$, a contradiction of the fact that $g \not\equiv 1 \pmod{q}$. Therefore, we have a short exact sequence
\[
1 \rightarrow H_1 \rightarrow G \stackrel{\phi}{\rightarrow}
\Z_p \ltimes \Z_q \rightarrow 1.
\]

We know that $f(\Z_p\ltimes\Z_q)=q$ by \propref{f-const pq group prop}. Therefore, by \propref{extension}, we have 
\[
d(G) \le (d(H_1)-1)f(\Z_p\ltimes \Z_q) + d(\Z_p \ltimes \Z_q) = (m-1)q + p + q -1 = n+p-1.
\]

Therefore, $d(G) = n+p-1$, and since equality holds, again by \propref{extension}, 
\[
f(G) \le f(H_1)f(\Z_p\ltimes\Z_q) = mq =n.
\]
However, $f(G)\ge \ord(x) = n \implies f(G)=n$.\\

\noindent \textit{Case II : ($q = p$)}\\
We write $k$ in place of $a_i$ for simplicity of notation. Since $g^p \equiv 1 \pmod{p^{k}}$ but $g \not\equiv 1 \pmod{p^{k}}$, we must have $k >1$. Let us write $n=mp^{k}$. Then, $H_1 := \langle x^{p^{k}} \rangle
\cong \mathbb{Z}_m$, and we have a short exact sequence
\[
1 \rightarrow H_1 \rightarrow G \stackrel{\phi}{\rightarrow}
\mathbb{Z}_p \ltimes \mathbb{Z}_{p^{k}} \rightarrow 1.
\]
(Again, if not, the images of $x$ and $y$ in $G/H_1$ would commute, which would imply that $xyH_1=x^gyH_1$ and hence, $x^{g-1} \in H_1= \langle x^{p^k} \rangle$, a contradiction of the fact that $g \not\equiv 1 \pmod{p^k}$). Now, we know that $f(\Z_p\ltimes\Z_{p^k})=p^k$ and $d(\Z_p\ltimes\Z_{p^k})=p^k+p-1$ by \propref{df-ppower}. Therefore, using \propref{extension}, we get 
\[
d(G) \le (d(H_1)-1)f(\Z_p\ltimes \Z_{p^{k}}) + d(\Z_p \ltimes \Z_{p^{k}}) = (m-1)p^k + p^k + p -1 = n+p-1.
\]

Therefore, $d(G) = n+p-1$, and since equality holds, again by \propref{extension}, 
\[
f(G) \le f(H_1)f(\Z_p\ltimes\Z_{p^k}) = mp^k =n.
\]
However, $f(G)\ge \ord(x) = n\implies f(G)=n$.
\end{proof}

\vskip 5mm

\section{Bounds for $\hol(\mathbb{Z}_p)$} \label{hol bound sec}

Let $p$ be an odd prime and consider the holomorph $\hol(\mathbb{Z}_p) =  \langle x,y \ | \ x^p=y^{p-1}=1, x^gy=yx \rangle$ where $\ord_p(g)=p-1$. Recall that, as part of \conjref{conj5}, we conjectured that, $f(\hol(\mathbb{Z}_p)) = p$.

We prove this for the primes $3, 5$. For a general prime
$p>5$, we are able to prove only the weaker result \thmref{holozp}.

We prove the theorem now for primes $p>5$. After that, we will verify the conjecture in the case of $p=3, 5$.

\begin{proof}[\textbf{ Proof of \thmref{holozp}}]
We have a prime $p \geq 7$. We write $G$ for $\hol(\mathbb{Z}_p)$ and $H= \langle x \rangle \cong \mathbb{Z}_p$. To show $e(\hol(\mathbb{Z}_p)) \leq p + \frac{p-3}{2}$, we show that given a product-one sequence $S$ of length $p+k$ (where $k \geq \frac{p-1}{2}$) in $G$, there exists a proper product-one subsequence of $S$.  

Let $S = (x^{a_1}y^{b_1}, \ldots, x^{a_{p+k}}y^{b_{p+k}})$ and look at the sequence $S_y = (b_1, \ldots, b_{p+k})$ in $\mathbb{Z}_{p-1}$. We write $\mathbb{Z}_{p-1}$ additively. The following two cases emerge:\\

\noindent {\it Case 1. ($S_y$ has a zero-sum subsequence of length $\leq k$.)}\\
In this case, fixing such a zero-sum subsequence of $S_y$, the sum of the
rest of the elements of $S_y$ must also be $0$ since $\Pi(S) \subseteq H$ implies
$\sum_{i=1}^{p+k} b_i =0$. Looking at
the corresponding sequence $S'$ of elements in $S$, we have $\Pi(S')
\subseteq H$ and $p \leq |S'| < p+k$. Therefore, by \lemref{usefulbass}$(a)$, $S'$ contains a product-one subsequence. So, we are done in this case.\\

\noindent {\it Case 2. (Case 1 does not hold.)}\\
That is, every zero-sum subsequence of $S_y$ has length $>k$. We construct disjoint subsequences $B_1, \ldots, B_k$ of $S_y$ where no $B_i$ has any repeating elements in the following manner. Start with $k$ elements of $S_y$ and make $k$ singleton sets with these elements. Choose each of the remaining elements of $S_y$ and, one by one, include them in one of these $k$ sets that do not already contain a copy of this element - if such a set is there (if no such set exists, then discard that element). In this manner, we have $k$ disjoint subsequences $B_i$ in $S_y$. If $0 \in B_i$ for some $i$, then we have nothing to prove as before. Otherwise, consider $B_1'=B_1$ and $B_i' = B_i \cdot (0)$ for $i>1$. Therefore, by our present assumption, $0$ is not in $\sum_{i=1}^k B_i' = B$. Considering the stabilizer $\Stab(B)$ of $B$, we get two subcases:\\

\noindent
{\it Case 2(a). ($\Stab(B) \neq 0$)}\\
Put $\Stab(B) = \langle c \rangle$ (written additively); we have $c |p-1$, and so,
$c \leq \frac{p-1}{2} \leq k$. If $c_i \in B_i$ for each $i \leq k$,
then there exists a subsequence of $c_1, \ldots, c_k$ whose sum is
divisible by $c$. Thus, there is a positive integer $b$ such that
$bc \in B$. But then $-bc \in \langle c \rangle = \Stab(B)$ which implies $0 = bc +
(-bc) \in B$.\\

\noindent {\it Case 2(b). ($\Stab(B) = (0)$)}\\
Using the DGM theorem (\lemref{DGM}), we obtain 
\[
|B| \geq  \min \left\{p-1, \sum_{i=1}^k |B_i'|-k+1\right\} = \min \left\{p-1, \sum_{i=1}^k |B_i|\right\}.
\]

As $0 \not\in B$ by assumption, we have $\sum_{i=1}^k
|B_i| \leq p-2$; so,
$$|S \setminus B_1 \cdots B_k| \geq (p+k)-(p-2)=k+2.$$
Since every element of $S\setminus B_1\cdots B_k$ is present in each
of $B_1, \ldots, B_k$, each of them must have at least $k+1$
copies in $S_y$.

We claim that there is only one such element.

For, if $a,b \in B_i$ for all $i$, then $\sum_{i=1}^k |B_i| \geq 2k
\geq p-1$, which is a contradiction to the inequality above. Therefore, $S \setminus B_1 \cdots B_k$ consists of copies of a single element $a$. Hence, the number of times $a$
appears in $S$ at least $k+(k+2)=2k+2$ times. We claim that
$(a, p-1)=1$; else, we can find a zero-sum subsequence of length
$\frac{p-1}{(a,p-1)}$ in $S_y$. This would lead to a contradiction,
as $\frac{p-1}{(a,p-1)} \leq \frac{p-1}{2} \leq k$. 

Now, as in \textit{Case II} of proof of \propref{f-const pq group prop}, considering the isomorphism $G\xrightarrow{\sim} \Z_{p-1}\ltimes_{g'}\Z_{p}$ (where $(g')^a\equiv g\pmod{p}$), defined by $x \mapsto x$ and $y\mapsto y^a$, we can assume, without loss of generality, that $a=1$. 

Therefore, we get a subsequence $S'\subset S$ of the form
\[
S' = \left(x^{a_1}y, \ldots, x^{a_{p+1}}y\right) \subset S.
\]
Now, we look at the sequence formed by the powers of $x$ of these elements, $S'_x:= (a_1,
\ldots, a_{p+1})$. We claim that in either of the following two
cases hold, we are done:

\vspace{2mm}
\noindent {\it Case (I). There exist $\frac{p-1}{2}$ disjoint pairs in $S'_x$.}\\
{\it Case (II). There exist two 3-element subsets in $S'_x$ which are disjoint as subsequences.}

\vspace{2mm}
To show that we are done if \textit{Case (I)} holds,
consider $(d_1,d_1), \ldots, \left(d_{\frac{p-1}{2}}, d_{\frac{p-1}{2}}\right)
\subset S'_x$. We have the product
\[
\left(x^{d_1}y \cdots x^{d_{\frac{p-1}{2}}}y\right)
\left(x^{d_1}y \cdots x^{d_{\frac{p-1}{2}}}y\right) = 1
\]
because the power of
$y$ is $y^{p-1}=1$ and the power of $x$ is, (for some $d$), equal to
$x^{d(1+g^{(p-1)/2})} = 1$. Thus, we are done if \textit{Case (I)} holds, since
$\left(x^{d_1}y, \ldots, x^{d_{\frac{p-1}{2}}}y, x^{d_1}y, \ldots, x^{d_{\frac{p-1}{2}}}y\right) \subset S.$

To see that we are done if \textit{Case (II)} holds, let $W_1=(c_1,c_2,c_3), W_3 = (e_1,e_2,e_3) \subset S'_x$, where $c_{i}$ and $e_{i}$ are distinct. Write
$$S'_x \setminus W_1 W_3 = (e_4, \ldots, e_{p-2}).$$
Take $W_i = W_{i-1}\cdot (e_i)$ for $4 \leq i \leq p-4$ and take $W'_i =
(x^ay : a \in W_i)$ for all $i$. Using Lemma \ref{3 dist elts lem}, we get $|\Pi(W_1')|, |\Pi(W'_3)| \geq 4$. Now, we apply Lemma
\ref{inductive lem} to induct on $i$. We can see that $|\Pi(W'_i)|
\geq i+1$ for $3 \leq i \leq p-4$; in particular, $|\Pi(W'_{p-4})|
\geq p-3$. This is because since each element of $\Pi(W_i')$ is of the form $x^dy^i$, we have $W_i' \supseteq \{x^{d_1}y^i, \ldots, x^{d_{i+1}}y^i \}$ by the induction hypothesis, and so,
\[
\Pi(W_{i+1}') \supseteq \left\{x^{d_j+g^ic_{i+1}}\right\}_{1 \leq j \leq i+1} \bigcup \left\{x^{c_{i+1}+gd_j}\right\}_{1 \leq j \leq i+1}.
\]

Then $W := W_1' W_{p-4}' \subset S$ satisfies $|W| = 3+p-4=p-1$ and (by Cauchy-Davenport inequality),
$|\Pi(W)| \ge \min\{p, 4+(p-3)+2-1\} = p$. Thus $1 = x^0 \in \Pi(W)$. Therefore, if either of these cases \textit{(I)} or \textit{(II)} holds, we are done.\\

To finish the proof of the theorem in \textit{Case 2}, we shall examine all the situations where neither \textit{Case (I)} nor \textit{Case (II)} holds.

Recall the set-up - there is no zero-sum subsequence of $S_y$ of
length $\leq k$, and we have a sequence $S'= (x^{a_1}y, \ldots,
x^{a_{p+1}}y) \subset S$. We have written $S'_x = (a_1, \ldots,
a_{p+1})$. We shall consider four cases according to how many
distinct elements $S'_x$ contains. \\
$\bullet$ If $S'_x$ contains $5$ or more distinct elements, then either there are at least $6$ distinct elements, or one of them must repeat because $p-1>5$. Thus, we are done by \textit{Case (II)} as shown above.\\
$\bullet$ If $S'_x$ contains at most $3$ distinct elements, let us pair equal ones until we are left with at most $3$ elements. This gives rise to at least $\lceil (p+1-3)/2 \rceil = \frac{p-1}{2}$ pairs. Thus, again \textit{Case (I)} shows that we are done.\\
$\bullet$ If $S'_x$ contains exactly $4$ distinct elements, say
$a,b,c,d$ and, if at least two of them, say $a,b$ are present
multiple times, we have $\{a,b,c\}$ and $\{a,b,d\}$ in $S'_x$, and \textit{Case (II)} again proves our contention.\\
$\bullet$ Finally, we are left with the case when $S'_x$ has $4$ distinct elements with one element repeated $p-2$ times (and, therefore, the others not repeating). So,
\[
S'_x = (b,c,d,a,a, \ldots, a)
\]
where $a$ occurs $p-2$ times, and $a,b,c,d$ are distinct. As $b+c, b+d$ are unequal, at least one of them, say $b+c$ $\neq 2a$. So, $\frac{b-a}{a-c} \neq 1$ and we get $0 < \theta < p-1$ satisfying $g^{\theta} = \frac{b-a}{a-c}$. Consider the $(p-1)$-fold product below, where $x^by$ appears in the first position, $x^cy$ appears at the $(\theta +1)$-th position, and all other terms are $x^ay$. We have
\[
x^by x^ay x^a y \cdots x^ay x^cy x^a y \cdots x^a y
= x^{b+ \sum_{i=2}^{\theta} g^{i-1}a + g^{\theta}c + \sum_{i=\theta+2}^{p-1} g^{i-1} a}=1
\]
because the exponent of $x$ is $\left(\sum_{i=1}^{p-1} g^{i-1} \right) a + b-a + g^{\theta}(c-a)=0$.
Therefore, even in this case we have obtained $1 \in
\Pi(x^ay, \ldots , x^a y, x^by, x^cy)$ where $x^ay$ is repeated $p-3$ times. The proof of the theorem for $p>5$ is now complete.
\end{proof}

\vskip 5 mm

Now we prove \conjref{conj5} for $p=5$.

\begin{prop}\label{holz5}
The conjecture $f(\hol(\mathbb{Z}_p)) = p$ holds for $p= 3, 5$.
\end{prop}

\begin{proof}
Since $\hol(\Z_3)\cong \mathbf{D}_6$, \thmref{dihdic} implies that $f(\hol(\Z_3))= 3$.

Now, we prove the conjecture for $p=5$. Now, $G =
\hol(\mathbb{Z}_5) = \langle x,y|x^5=y^4=1, x^2y=yx \rangle$. Note that $d(G)= 8$ \cite{JB07}. We take $H = \langle x \rangle \cong \mathbb{Z}_5$, and write $H^{\ast} = H \setminus \{1\}$. As
$x$ has order $5$, we have $e(G) \geq f(G) \geq 5$. Let $S$ be a
sequence of length $8$ in $G$. Denote by $S_y$, the sequence of
powers of $y$ of the elements of $S$; say, $S_y = (a_1, \ldots,
a_8)$. This is a subsequence of length $8$ in $\mathbb{Z}_4$, which we write additively. Define $\Sigma$ to be the set of sums of any $5$ elements of $S_y$. 

Now, if $0 \in \Sigma$, then considering the corresponding elements
in $S$, we have a sequence $T \subset S$ such that $|T|=5$ and
$\Pi(T) \subset H$. Then \lemref{usefulbass}$(a)$ gives us a
product-one subsequence of $T$, and we are done in this case.

Consider the general situation now. For $0 \leq i \leq 2$, let
$n(i)$ denote the number of times $i$ appears in $S_y$ (could be $0$
also). We have several cases.\\
$\bullet$ If $n(0) \neq 0$, let $S' = S_y\setminus(0)$. Since $|S'|=7$, by \cite{EGZ}, we get a length $4$ subsequence of $S'$ whose sum is $0$; this gives a length $5$ subsequence of $S_y$ whose sum is $0$. So, we are done in this case.\\
$\bullet$ If $n(2) \geq 6$, we have then $6$ elements of $S$ in the
subgroup $ \langle x,y^2 \rangle \cong \mathbf{D}_{10}$. Since this group has $d(\mathbf{D}_{10})=6$ and
$f(\mathbf{D}_{10})=5$, as we have proved earlier, we are again done.\\
$\bullet$ If $3 \leq n(2) \leq 5$, then $n(1)+n(3) \geq 3$. Thus, we
either have two copies of $1$ or two copies of $3$. Together with
three copies of $2$, we have a length $5$ subsequence of $S_y$ whose
sum is $0$ in $\mathbb{Z}_4$. Therefore, we are done in this case also.\\
$\bullet$ If $1 \leq n(2) \leq 2$ and both $n(1), n(3)>0$, then one
of $1$ and $3$ appears thrice at least. Again, we can get a length
$5$ subsequence in $S_y$ summing to $0$. Since $1+1+1+3+2$ and
$3+3+3+1+2$ are both $0 \pmod{4}$, we are done in this case as well.\\
$\bullet$ Thus, we may assume either $n(2)=0$ or $n(1) \geq 6$ or
$n(3) \geq 6$.

If the sequence $S^{-1}$ of inverses of elements of the sequence $S$ has a product-one subsequence of length $\le 5$, then $S$ also has one such subsequence. Thus, we may assume without loss of generality that $n(1) \geq
n(3)$. Now, we have several cases.\\

\noindent
{\it Case I.  ($n(1) \ge 6$)}.\\
We then have $6$ elements $x^{a_i}y$ with $a_i \in \mathbb{Z}_5$ ($1
\leq i \leq 6$). If we look at any $4$ of them, then the product
looks like
\[
x^ay x^by x^cy x^dy = x^{a+2b-c-2d}.
\]

Among the six $a_i$s, if we have two equal pairs, we get a length $4$ subsequence of $S$ whose product is $1$ by the above formula. We are done in this case. In the contrary case, the six $a_i$s must have at least four distinct values; say, $a,a+1,a+2,a+3$. Now
$$x^ay x^{a+1}y x^{a+3}y x^{a+2}y = x^{-5}=1.$$
Since this also leads to a product-one subsequence of $S$ of length $4$, \textit{Case I} is done.\\

\noindent
{\it Case IIa. ($n(1)=4$ and $n(2)=0$)}\\
Then $n(3)=4$. Among elements of $S$ of the form $x^ay$, if either
two equal pairs occur, or all four are distinct, we are done as in \textit{Case I}. Therefore, consider the only other possibility, when the
elements are $x^ay, x^ay, x^by, x^cy$ where $a \neq b$ ($a=c$ is a possibility). By multiplying two among them, we may obtain the
distinct elements $x^{3a}y^2, x^{2a+b}y^2, x^{a+2b}y^2$. Hence, we have obtained three distinct elements from the
assumption that $n(1)=4$. Similarly, from $n(3)=4$, we have four
elements of the form $x^uy^3, x^u y^3, x^v y^3, x^ry^3$ form which
multiplication of pairs produces three distinct elements of the form
$x^ty^2$. These six elements have the same power of $y$ (viz.,
$y^2$) and hence, cannot be all different because $x$ has only five
possible, different exponents. Thus, we have an element $x^ty^2$
obtained from multiplying two terms from $x^ay, x^ay, x^by, x^cy$ and also
obtained from multiplying two terms of the form $x^u y^3, x^u y^3, x^v y^3, x^ry^3$. Since
$x^ty^2 x^ty^2 = 1$, we have a product-one subsequence of $S$ of length $4$. Therefore, we are done in \textit{Case IIa} also.\\

\noindent  {\it Case IIb. ($n(1)=5$ and $n(2)=0$)}\\
Then $n(3)=3$. Once again, as in \textit{Case IIa}, if we have either four
distinct elements or two equal pairs among the elements of the form $x^ay$, we are done. So, we look at the only other possibility
$$x^ay, x^ay, x^ay, x^by, x^cy,$$
where $a, b$ and $c$ are distinct. The two-fold products produce the elements $x^{t}y^2$ for
\[
t = 3a, a+2b, a+2c, b+2a, c+2a, b+2c, c+2b.
\]

We claim that either the five elements corresponding to $t=3a, a+2b,
b+2a, a+2c, c+2a$ are distinct mod $5$, or we have a product-one
subsequence of length $4$ in $S$. Indeed, as $a,b,c$ are distinct
mod $5$, the five elements above are distinct unless either $a+2b =
2a+c$ or $a+2c = 2a+b$. 

If $a+2b=2a+c$, then $x^ay x^by x^c y x^ay = x^{a+2b-c-2a}=1$ and we
are done. Similarly, if $a+2c=2a+b$, then $x^ay x^cy x^by x^ay  = 1$.

Otherwise, we have five distinct powers of $x$ which gives
\[
\left\{x^{3a}y^2,x^{a+2b}y^2, x^{a+2c}y^2, x^{2a+b}y^2, x^{2a+c}y^2\right\} = \left\{x^iy^2: 0 \leq i \leq 4\right\}.
\]

Now, if we consider any two elements of the form $x^cy^3$ (note that $n(3)=3$) in $S$, their product is of the form $x^dy^2$, and hence, must be in the above set. Elements of the set are their own inverses. Hence, there is a length $4$ subsequence of $S$ whose product is $1$. Therefore, we are done in \textit{Case IIb} also. The proof of the conjecture is complete for $p=5$.
\end{proof}

We end this section by proving \propref{ebound}, thereby establishing a non-obvious lower bound for $e(\hol(\Z_p))$.

\begin{proof}[Sketch of proof of \propref{ebound}]
    Since $f(\hol(\Z_5)) < d(\hol(\Z_5))$, we have that $e(\hol(\Z_5)) < d(\hol(\Z_5)) = 8$. If a minimal product-one sequence $S$ of elements of $\hol(\Z_5)$ of size $7$ exists, focus on $S_y$, the sequence of powers of $y$ of elements of $S$. Note that, as in the proof of \propref{holz5}, if there is a subsequence of $S_y$ of size $\le 2$ which adds to $0$, it will contradict the minimality of $S$. To avoid that, $S_y$ must contain six copies of $1$ or six copies of $3$. But in that case, we get a contradiction by \textit{Case I} from the proof of \propref{holz5}. Therefore, $e(\hol(\Z_5))\le 6$. So, it is enough to show that $e(\hol(\Z_p)) \ge p+1\,\forall\, p\ge 5$.
    
    For any prime $p\ge 5$, consider the following sequence of size $p+1$:
    \[
    S=(y, \underbrace{xy, \ldots, xy}, x^a y^{\frac{p-1}{2}}, x^b y^{\frac{p-1}{2}} )
    \]
    where $xy$ occurs $p-2$ times, and $a\neq b$ are picked from $\Z_p\setminus C$, where $C$ is the set of indices $c$ such that $x^{c}y^{\frac{p-1}{2}}$ can be obtained as a product of some subsequence of $(y, xy, \ldots, xy)$, where $xy$ appears $p-2$ times. Note that $|C| = \frac{p+1}{2}$, where one term comes from $(xy)^{\frac{p-1}{2}}$, and the other $\frac{p-1}{2}$ terms come from the cyclic products of $(y, xy, \ldots, xy)$, where $xy$ appears $\frac{p-3}{2}$ times. Therefore, $|\Z_p\setminus C| = \frac{p-1}{2} \ge 2$.

    Since $|S|>p$ and $\Pi(S)\subseteq\langle x \rangle$, by \lemref{usefulbass}$(a)$, $S$ has a product-one subsequence. Therefore, if we show that $S$ does not have a proper subsequence which is product-one, then $S$ is a minimal product-one sequence, which proves that $e(\Z_p)\ge p+1$.

    Note that, $S_y = \left(1, 1, \ldots, 1, \frac{p-1}{2}, \frac{p-1}{2}\right)$, and $y$ powers of any product-one subsequence adds up to $0$. The only such proper subsequences of $S_y$ are $\left(\frac{p-1}{2}, \frac{p-1}{2}\right)$ and $\left(1,1,\ldots, 1, \frac{p-1}{2}\right)$. Since $a\neq b$ and $\left(x^cy^{\frac{p-1}{2}}\right)^{-1} = x^cy^{\frac{p-1}{2}}$, the subsequence corresponding to $\left(\frac{p-1}{2}, \frac{p-1}{2}\right)$ is not product-one. Taking any sequence corresponding to $\left(1, 1, \ldots, 1, \frac{p-1}{2}\right)$, if product under any permutation is $1$, then all cyclic products corresponding to that permutation are $1$ as well, so we can assume the term $x^d y^{\frac{p-1}{2}}$ appears at the end. Now, the rest of the product is either $(xy)^{\frac{p-1}{2}}$, or some cyclic product of $(y, xy, \ldots, xy)$. Since we chose $a, b\notin C$, the product can not be $1$. Therefore, we see that $S$ does not have a proper subsequence which is product-one. So, we are done.
\end{proof}

\begin{remark}
    We see that $d(\hol(\Z_5))= 8$ \cite{JB07}, $e(\hol(\Z_5))= 6,$ and $f(\hol(\Z_5))=5$, so the inequalities $d(G)\ge e(G)\ge f(G)$ are all strict for $G=\hol(Z_5)$.   
\end{remark}
\vskip 5mm

\section{Computing $d_k(G)$}

The constant $f(G)$ has certain properties that provide more information on the Davenport constant and some of its variants. We have seen one such result in \propref{extension}.
Recall the definition of $d_k(G)$ from \defref{d_k}. It is easy to see that $d_k(G) \leq k d(G)$ for all $k$, but using $f(G)$, we can strengthen this inequality. The resulting statement is that of \propref{dkintermsoff}, which we prove now.

\begin{proof}[\textbf{Proof of \propref{dkintermsoff}}]
Let $S$ be a sequence in $G$ with $|S| = (k-1)f(G) + d(G)$. Let us keep removing disjoint product-one subsequences $S_1, \ldots, S_t$ of length $\leq f(G)$ until no such subsequences remain. Then, by the definition of $d(G)$ and $f(G)$, we have
\[
|S \setminus S_1 \cdots S_t| < d(G).
\]
Therefore, $d(G) > |S| - \sum_{i=1}^t |S_i| \geq |S| - tf(G)=
(k-1-t)f(G) + d(G)$. So, we get $t \geq k$. Thus we get $k$ disjoint product-one subsequences of $S$, and the proof is complete.
\end{proof}

As a consequence of this proposition, we can compute the $k$-th Davenport constant of several groups. We compile them in the following corollary.

\begin{cor} For any prime $p$, and positive integers $k, n$, we have
    \begin{enumerate}[label=(\alph*)]
        \item $d_k(\mathbf{D}_{2n}) = nk + 1$.
        \item $d_k(\mathbf{Q}_{4n}) = 2nk + 1$.
        \item $d_k(\Z_p\ltimes\Z_n) = nk + p -1$.
    \end{enumerate}
\end{cor}

\begin{proof}
    By \propref{dkintermsoff}, we have that
    \begin{align*}
        &d_k(\mathbf{D}_{2n}) \le (k-1) f(\mathbf{D}_{2n}) + d(\mathbf{D}_{2n}) = (k-1)n + n + 1 = nk +1.\\
        &d_k(\mathbf{Q}_{4n}) \le (k-1) f(\mathbf{Q}_{4n}) + d(\mathbf{Q}_{4n}) = (k-1)2n + 2n + 1 = 2nk +1.\\
        &d_k(\Z_p\ltimes\Z_n) \le (k-1) f(\Z_p\ltimes\Z_n) + d(\Z_p\ltimes\Z_n) = (k-1)n + n + p - 1 = nk + p - 1.
    \end{align*}
    
    The lower bounds are easy to obtain using the following sequences: $nk- 1$ many copies of $x$ and a single copy of $y$ for $\mathbf{D}_{2n}$; $2nk-1$ many copies of $x$ and a single copy of $y$ for $\mathbf{Q}_{4n}$; $nk - 1$ many copies of $x$ and $p-1$ many copies of $y$ for $\Z_p\ltimes\Z_n$ (recall the descriptions of $\mathbf{D}_{2n}, \mathbf{Q}_{4n}$ and $\Z_p\ltimes\Z_n$ from \S~\ref{notation}).
\end{proof}

\bigskip

\noindent \emph{\textbf{Acknowledgments.}}
The first, second, and third authors would like to thank the Indian Statistical Institute, Bangalore centre for providing an ideal environment to carry out this work. The second and third author express their gratitude to NBHM for financial support during the period of this work.

\bibliographystyle{plain}  

\end{document}